\documentclass[11pt]{article}
\usepackage[a4paper]{anysize}\marginsize{3.1cm}{3.1cm}{0.9cm}{3cm}
   \pdfpagewidth=\paperwidth \pdfpageheight=\paperheight
\usepackage{amssymb,amsthm}

\newtheoremstyle{plain}{\topsep}{\topsep}{\itshape}{9mm}{\bfseries}{.}{0.9em}{}
\newtheoremstyle{definition}{\topsep}{\topsep}{}{9mm}{\bfseries}{.}{0.9em}{}

\theoremstyle{plain}
\newtheorem{thm}{Theorem}[section]
\newtheorem{theorem}[thm]{Theorem}
\newtheorem{lemma}[thm]{Lemma}

\newtheorem{conjecture}[thm]{Conjecture}

\theoremstyle{definition}

\newtheorem{remark}[thm]{Remark}
\newtheorem{example}[thm]{Example}

\newtheorem{question}[thm]{Question}

\newenvironment{varthm*}[1]{\trivlist\item[\hskip\parindent]{\bf #1.}\it}{\endtrivlist}

\newcommand\adj{_{\rm adj}}
\renewcommand\ge{\geqslant}
\renewcommand\geq{\geqslant}
\renewcommand\le{\leqslant}
\renewcommand\leq{\leqslant}

\newcommand\tfrac[2]{{\textstyle\frac{#1}{#2}}}
\newcommand\abs[1]{\left|#1\right|}
\newcommand\be{\begin{eqnarray*}}
\newcommand\ee{\end{eqnarray*}}
\newcommand\compact{\itemsep=0cm \parskip=0cm}
\makeatletter\newcommand\xcompact{\itemsep=0cm \parskip=0cm \@topsep=0pt \@topsepadd=-\@outerparskip}\makeatother
\newcommand\Q{\mathbb Q}
\newcommand\R{\mathbb R}
\newcommand\Z{\mathbb Z}
\newcommand\N{\mathbb N}
\renewcommand\P{\mathbb P}
\newcommand\eqnref[1]{(\ref{#1})}
\newcommand\engqq[1]{``#1''}

\newcommand\grant[1]{{\renewcommand\thefootnote{}\footnotetext{#1.}}}
\newcommand\keywords[1]{{\renewcommand\thefootnote{}\footnotetext{\textbf{Keywords:} #1.}}}
\newcommand\subclass[1]{{\renewcommand\thefootnote{}\footnotetext{\textbf{Mathematics Subject Classification (2010):} #1.}}}

\begin{document}

   \title{On the boundedness of the denominators \\ in the Zariski decomposition on surfaces}
   \author{\normalsize Thomas Bauer, Piotr Pokora, David Schmitz}
   \date{\normalsize October 4, 2015}
   \maketitle
   \thispagestyle{empty}
   \grant{%
      The first and third author were supported by DFG grant BA 1559/6-1.
      The second author was supported by DAAD}
   \subclass{14A25, 14C20
      }

\begin{abstract}
   Zariski decompositions play an important role in the theory of
   algebraic surfaces.
   For making geometric use of the decomposition of a given divisor,
   one needs to pass to a multiple
   of the divisor in order to clear denominators.
   It is therefore
   an intriguing question whether the surface has a
   \engqq{universal denominator} that can be used to simultaneously
   clear denominators in all
   Zariski decompositions on the surface.
   We prove in this paper that, somewhat surprisingly, this condition of
   \emph{bounded Zariski denominators} is equivalent to
   the
   bounded negativity of curves that is addressed in the
   \emph{Bounded Negativity Conjecture}.
   Furthermore, we provide explicit bounds for Zariski denominators
   and negativity of curves in terms of each other.
\end{abstract}


\section*{Introduction}

   The theorem on Zariski decomposition is a
   fundamental tool in the theory of
   algebraic surfaces.
   It was
   established by Zariski \cite{Zar62} for effective divisors
   and
   extended by Fujita \cite{Fuj79} to the pseudo-effective case.
   The geometric significance
   of Zariski decompositions
   lies in the fact that,
   given a pseudo-effective integral divisor $D$ on $X$ with
   Zariski decomposition $D=P+N$, one has
   for every sufficiently divisible integer $m\ge 1$ the equality
   $$
      H^0(X,\mathcal O_X(mD)) = H^0(X, \mathcal O_X(mP))
      \,.
   $$
   In other words,
   all sections of
   $\mathcal O_X(mD)$ come from
   the nef line bundle $\mathcal O_X(mP)$.
   The term \engqq{sufficiently divisible} here means that
   one
   needs to
   pass to a multiple $mD$ that clears
   denominators in $P$ for the statement to
   hold.
   Of course, it would be most pleasant if
   one knew --
   beforehand and independently of $D$ -- which multiple to take.
   This amounts to asking the following

\begin{varthm*}{Question}
   \rm
   Let $X$ be a smooth projective surface.
   Does there
   exist an integer $d(X)\ge 1$ such that
   for every pseudo-effective integral divisor $D$
   the denominators in the Zariski decomposition of $D$
   are bounded
   from above by $d(X)$?
\end{varthm*}

\noindent
   If such a bound $d(X)$ exists, then
   we say that $X$ \emph{has bounded Zariski
   denominators}.
   Taking then the factorial $d(X)!$, one has in fact
   a uniform number that clears denominators in all
   Zariski decompositions on $X$.
   It is an intriguing question as to whether
   a given
   smooth surface satisfies this boundedness
   condition.

   We show in the present paper that, somewhat surprisingly,
   boundedness of Zariski denominators is equivalent to bounded
   negativity:

\begin{varthm*}{Theorem}
   For a smooth projective surface $X$ over an algebraically
   closed field the following two
   statements are equivalent:
   \begin{itemize}
   \item[\rm(i)]
      $X$ has bounded Zariski denominators.
   \item[\rm(ii)]
      $X$ has bounded negativity, i.e., there is a bound $b(X)$
      such that for every irreducible curve $C$ on $X$ one has
      $C^2 \ge -b(X)$.
   \end{itemize}
\end{varthm*}

   The \emph{Bounded Negativity Conjecture} (BNC), explored in
   \cite{Duke}, is the conjecture that (ii) is true for every
   smooth projective surface over the field of complex numbers.
   As mentioned in \cite{Duke}, the exact origin of the
   conjecture is unclear, but it has a long oral tradition
   that can be traced back
   via Ciro Ciliberto and Alfredo Franchetta
   to Federigo Enriques.
   The conjecture is open in general, even for
   the case where $X$ is the blow-up of $\mathbb{P}^2$ in $s$
   general points with $s \geq 10$ (see \cite{Harbourne1} for a
   nice introduction to this subject).
   By contrast, it is known that bounded negativity
   does not hold in general in positive characteristics~--
   this is in accordance with
   Example~\ref{ex:positive-char}, where we exhibit
   unbounded Zariski denominators on such surfaces.

   Our result
   sheds new light on BNC:
   It says in particular that BNC is equivalent to
   boundedness of Zariski denominators on all smooth complex projective
   surfaces.
   We do not dare to suggest whether this makes it
   more likely or less likely that BNC holds~--
   it definitely makes it more desirable to hold.

   On the practical side, we provide explicit
   bounds for $d(X)$ and $b(X)$ in terms of each
   other (see Theorems~\ref{BNC-to-denominators} and \ref{denominators-to-BNC}).
   If, for instance,
   all negative curves on $X$ are known,
   then one has
   an effective bound on the Zariski denominators.

\paragraph*{\it Acknowledgement.}
   We thank Alex K\"uronya for sharing with us the question
   that motivates this note.
   Also, we would like to thank
   S.~M\"uller-Stach and T.~Szemberg for helpful comments.
   We would like to thank the referee for valuable suggestions concerning
   Example~\ref{example:large-d}.


\section{Denominators in Zariski decompositions}

   Let $X$ be a smooth projective surface and $D$ a
   pseudo-effective integral divisor on $X$.
   Fujita's extension \cite{Fuj79} of Zariski's result
   \cite{Zar62}
   states that $D$ can be
   written uniquely as a sum
   $$
      D = P + N
   $$
   of $\Q$-divisors such that
   \begin{itemize}\compact
   \item[(i)]
      $P$ is nef,
   \item[(ii)]
      $N$ is effective and has negative
      definite intersection matrix if $N\ne 0$,
   \item[(iii)]
      $P\cdot C=0$ for every component of $N$.
   \end{itemize}
   For the question of bounded denominators in $P$ and $N$ it
   is of course enough to consider the denominators of
   $N=\sum_{i=1}^k a_iN_i$,
   i.e., the denominators of the coefficients $a_i$.
   In order to approach the problem, we use the following description of
   the coefficients $a_i$.
   They are given as the (unique) solution of the system of equations
   $$
      D\cdot N_{j} = (P + \sum_{i=1}^{k} a_{i}N_{i})\cdot N_{j} = \sum_{i=1}^{k} a_{i}N_{i}\cdot N_{j}
      \qquad\mbox{for all $j \in \{1, \dots, k\}$}.
   $$
   This system  can be rewritten in  matrix form as
   $$
      S[a_{1}, \dots, a_{k}]^{t} = [D\cdot N_{1}, \dots,  D\cdot N_{k}]^{t}
      \,,
   $$
   where $S$ denotes the intersection matrix of the curves $N_{1},\dots,N_k$, i.e.
   $S = [N_{i}\cdot N_{j}]_{i,j} \in M_{k \times k}(\mathbb{Z})$.
   Since the matrix $S$ is negative definite, it has non-zero determinant, and using
   Cramer's rule one has
   \begin{equation}\label{eqn:ai}
      a_{i} = \frac{ \det [s_{1}, \dots, s_{i-1}, b, s_{i+1}, \dots, s_{k}] } {\det(S)}
      \,,
   \end{equation}
   where $s_{i}$ denotes the $i$-th column of the matrix $S$ and $b = [D\cdot N_{1}, \dots,  D\cdot N_{k}]^{t}$.
   Thus, for divisors with negative part $N$ supported on $N_1,\dots,N_k$,
   the denominators of the Zariski decomposition are bounded by $\abs{\det(S)}$.

\begin{remark}
   Note that the above reasoning yields an upper bound for the
   denominators of the coefficients in the Zariski decomposition for
   any surface whose pseudo-effective cone is rational polyhedral,
   since in this case there are only finitely many possible sets $\{N_1,\dots,N_k\}$
   of components of negative parts, so we obtain the bound
   \be
      d(X) = \max \{ \abs{\det(S_i)} \mid S_i \mbox{ principal negative definite submatrix of } S\}
      \,,
   \ee
   where $S$ denotes the intersection matrix of all irreducible curves with
   negative self-intersection.

   It is not clear a priori whether the corresponding supremum will be finite
   in the presence of infinitely many extremal rays. This more general
   situation is the topic of the following section.
\end{remark}


\section{Bounded denominators and bounded negativity}

   We start by reminding the reader of the following conjecture.
   \begin{conjecture}[Bounded Negativity Conjecture, \cite{Duke}, \cite{Harbourne1}]
      \label{BNC}
      Let $X$ be a smooth complex projective surface.
      Then there exists an integer $b({X})\ge 0$ such that for every irreducible and reduced curve $C$ one has
      $C^{2} \geq - b({X})$.
   \end{conjecture}

   The aim of this section is to prove
   the theorem stated in the introduction.
   In particular, we thus show that
   Conjecture \ref{BNC} is equivalent to
   the assertion that all smooth complex projective surfaces have
   bounded
   Zariski denominators.

\begin{theorem}\label{BNC-to-denominators}
   Let $X$ be a smooth projective surface on which the self-intersection
   of irreducible curves is bounded by $-b(X)$. Then $X$ has bounded Zariski denominators.
   More concretely, denoting by $\rho(X)$ the Picard number, we have
   $$
      d(X) \le b(X)^{\rho(X) -1}
      \,.
   $$
\end{theorem}

\begin{proof}
   Let $D$ be any integral pseudoeffective divisor on $X$, with
   negative part $N=\sum a_iN_i$, $a_i>0$, and let $S$ be the
   intersection matrix of the curves $N_1,\ldots,N_k$.
   According to the consideration in the first section, we know that the
   denominators of the $a_i$ can be at most
   $\abs{\det \, {S}}$.

   Since $S$ is negative definite, there exists an invertible matrix $U \in GL(k, \mathbb{R})$
   and real numbers $\lambda_1,\dots,\lambda_k$
   such that $U^{-1}SU = {\rm diag}(\lambda_{1}, \dots, \lambda_{k})$ and
   $\lambda_{i} < 0$ for each $i \in \{1, \dots, k\}$.  As the trace
   of a matrix is invariant under conjugation, we have
   $$
      N_1^2+\ldots+N_k^2 = {\rm tr}(S) = {\rm tr}(U^{-1}SU) = \lambda_{1} + \ldots + \lambda_{k}
      \,.
   $$
   The same holds for the determinant, thus
   $$
      |\det S| = |\det (U^{-1} S U)| = |\lambda_{1} \cdot \ldots \cdot \lambda_{k}| =
      |\lambda_{1} | \cdot \ldots \cdot |\lambda_{k}|
      \,.
   $$
   Using the inequality of arithmetic and geometric means we obtain
   $$
      |\det S| = |\lambda_{1} | \cdot \ldots \cdot |\lambda_{k}| \leq \bigg(\frac{ \sum_{i} |\lambda_{i}|}{k}\bigg)^{k} = \bigg(\frac{ - {\rm tr }(S) }{k}\bigg)^{k} = \bigg(\frac{ - N_1^2-\ldots -N_k^2}{k}\bigg)^{k}
      \,.
   $$
   By assumption, the self-intersection $N_{i}^{2}$ is at least $-b(X)$ for each $i \in \{1, \dots, k\}$,
   hence
   $$
      |\det S|\le \bigg(\frac{ - N_1^2-\ldots -N_k^2}{k}\bigg)^{k}\le \bigg(\frac{k\cdot b(X)}{k}\bigg)^k.
   $$
   Finally, by the Hodge Index Theorem, $k$ can be at most $\rho(X)-1$, thus
   $d(X) = b(X)^{\rho(X)-1}$ is a bound for the Zariski denominators of integral pseudo-effective
   divisors on $X$, which is independent of the particular support of the negative part.
\end{proof}

   We now turn to the converse implication.

\begin{theorem}\label{denominators-to-BNC}
   Let $X$ be a smooth projective surface. If
   Zariski denominators on $X$ are bounded by $d(X)$, then $X$ has
   bounded negativity.
   More concretely, denoting by $\Delta$ the discriminant
   of the N\'eron-Severi lattice $N^1(X)$
   (i.e., the determinant of the intersection form),
   we have
   $$
      b(X) \le d(X)\cdot d(X)! \cdot\abs\Delta
      \,.
   $$
\end{theorem}

   The proof rests on the following lemma.

\begin{lemma}\label{lemma:bounded}
   Let $X$ be a smooth projective surface with bound $d(X)$ of Zariski
   denominators.
   Then there exists for every
   negative curve $C$ on $X$
   an (integral) ample line bundle $A$ on $X$
   such that the gcd of the numbers $C^2$ and $A\cdot C$
   is a divisor of $d(X)!\cdot \Delta$.
\end{lemma}

   Granting the lemma, we give the proof of the theorem.

\begin{proof}[Proof of the theorem]
   Let $C$ be any irreducible negative curve on $X$. By Lemma
   \ref{lemma:bounded}, there exists an ample divisor $A$ such that
   $$
      \gcd(C^2,AC) \mid d(X)!\cdot\Delta
      \,.
   $$
   For a sufficiently large integer $k$,
   the line bundle $A + kC$ has Zariski decomposition
   $$
      A + k C = (A + \alpha C) + (k-\alpha)C
   $$
   where the divisor in parentheses is the positive part and
   $$
      \alpha=-\frac{A\cdot C}{C^2}
      \,.
   $$
   Now, the denominator of $\alpha$ is exactly
   $\tfrac{-C^2}{\gcd(C^2,AC)}$. In particular, by the assumed boundedness of the
   denominators,
   $$
      d(X) \ge \frac{-C^2}{\gcd(C^2,AC)} \ge \frac{-C^2}{d(X)!\cdot |\Delta|}
      \,.
   $$
   Therefore, the self-intersection of any curve $C$ is bounded
   from below by $-d(X)\cdot d(X)! \cdot |\Delta|$.
\end{proof}

   Now we turn to the proof of Lemma~\ref{lemma:bounded}. The first step is:

\begin{lemma}\label{lemma:divides-Delta}
   Let $X$ be a smooth projective surface with Zariski
   denominators bounded by $d(X)$. If $C$ is a negative curve and $t$
   an integer that divides $A\cdot C$ for all ample line bundles
   $A$ on $X$,
   then $t$ is a divisor of $d(X)!\cdot\Delta$.
\end{lemma}

\begin{proof}
   Let $F$ be the minimal integer divisor class in the ray in $N^1(X)_\R$
   spanned by $C$. Then $C=kF$ for some integer $k\ge1$, and $F$ is pseudo-effective
   with Zariski decomposition
   $$
      F = \frac1k \cdot C
      \,.
   $$
   By the boundedness assumption of Zariski denominators we have $k\le d(X)$.

   Note that if the hypothesis of the lemma is satisfied,
   then
   \begin{itemize}
   \item[(*)]
      $t$ divides $D\cdot C$ for all (integral) divisors $D$ on
      $X$.
   \end{itemize}
   In fact, if $D$ is any (integral) divisor, then
   $mA+D$ is ample for sufficiently large integers $m$, and hence
   $t$ divides both $mA\cdot C$ and $(mA+D)\cdot C$, so that it
   also divides $D\cdot C$.

   Choosing now a lattice basis of $N^1(X)$, we may think
   of the
   intersection form as given
   -- after modding out torsion --
   on $\Z^{\rho(X)}$ by an integral
   matrix $S$ of determinant $\Delta$, and of classes $C$, $D$ in
   $N^1(X)$ as
   represented by vectors $c$, $d$ in $\Z^{\rho(X)}$, with
   $$
      C\cdot D=c^t S d
      \,.
   $$
   In these terms, condition (*) implies that every entry in the
   vector $c^tS$ is divisible by $t$. Using now the adjugate
   matrix $S\adj$ of $S$, we infer that the vector
   $$
      c^tSS\adj=c^t\det(S)=c^t\Delta
   $$
   is divisible by $t$. Representing $F$ as an integral vector $f$,
   we obtain that $t$ is a divisor of
   $$
      f^t k \Delta
      \,.
   $$
   Now this implies that in fact $k \Delta$ is
   divisible by $t$, since otherwise every entry in the vector
   $f$ would be divisible by $t$, which in turn would mean
   that the class of $F$ is not primitive.
   Since $k\le d(X)$, in particular $t$ is a divisor
   of $d(X)!\cdot\Delta$.
\end{proof}

   We now prove Lemma~\ref{lemma:bounded} using
   Lemma~\ref{lemma:divides-Delta}.

\begin{proof}[Proof of Lemma~\ref{lemma:bounded}]
   Let $C$ be any negative curve on $X$.
   Assume by way of contradiction that the conclusion of the lemma
   is false.
   By the factorization theorem, then the following holds:
   \begin{itemize}
   \item[(+)]
      For every ample divisor $A$ on $X$ there exists a prime power $p^r$ such that
      $$
         p^r \mid C^2, \quad p^r \mid A\cdot C, \quad p^r \nmid d(X)!\cdot\Delta
         \,.
      $$
   \end{itemize}
   Note that there are only finitely many possibilities for prime
   powers satisfying (+), namely those $p^r$ that divide $C^2$ and do not divide $d(X)!\cdot\Delta$.
   Let $p_1,\dots p_s$ be the prime factors of $C^2$ such that there exists a
   power which divides $C^2$ but not $d(X)!\cdot\Delta$, and denote for each $i$ by
   $n_i$ the smallest number such that $p_i^{n_i}$ divides $C^2$ but not
   $d(X)!\cdot\Delta$.

   We claim:
   \begin{itemize}
   \item[(++)]
      There exists an $i\in\{1,\dots, s\}$ such
      that
      for all ample
      line bundles $A$, the intersection number
      $A\cdot C$ is divisible by $p_i^{n_i}$.
   \end{itemize}
   If (++) does not hold,
   then there is for every $i$ an ample
   divisor $A_i$ such that $p_i^{n_i}$ does not divide $A_i\cdot C$.
   Consider now the ample line bundle
   $$
      A:=p_2^{n_2}\dots p_r^{n_r} A_1+p_1^{n_1}p_3^{n_3}\dots p_r^{n_r} A_2
      +\dots+p_1^{n_1}\dots p_{r-1}^{n_{r-1}}A_r
      =\sum_{i=1}^r \frac{p_1^{n_1}\dots p_r^{n_r}}{p_i^{n_i}} A_i
      \,.
   $$
   By assumption (+), both $C^2$ and $A\cdot C$ are divisible by some $p_i^{r}$ that does not divide
   $d(X)!\cdot\Delta$. Therefore, by the minimality of the $n_i$ we have $r\ge n_i$, and $A\cdot C$ is
   also divisible by $p_i^{n_i}$. We can
   assume $i$ to be $1$. Now, $p_1^{n_1}$ divides all terms in the sum
   $$
      p_2^{n_2}\dots p_r^{n_r} A_1\cdot C+p_1^{n_1}p_3^{n_3}\dots p_r^{n_r} A_2\cdot C
      +\dots+p_1^{n_1}\dots p_{r-1}^{n_{r-1}}A_r\cdot C
   $$
   except for possibly the first one, and it divides the sum
   (which is $A\cdot C$).
   It must therefore also
   divide the first one, and hence it divides $A_1\cdot C$,
   which is a contradiction with the choice of $A_1$.

   We conclude that (++) holds.
   So the number $p_i^{n_i}$ that we have found divides $A\cdot C$
   for all ample line bundles $A$,
   thus by Lemma~\ref{lemma:divides-Delta} it divides $d(X)!\cdot\Delta$.
   Now, this is a contradiction with the choice of $n_i$, thus (+) is false
   and the lemma follows.
\end{proof}


\section{Examples}

\begin{example}[Surfaces with unbounded Zariski denominators in positive characteristic]\label{ex:positive-char}
   Let $C$ be a curve of genus $g\ge 2$ defined over a finite field of
   characteristic $p>0$.
   The surface $X=C\times C$ is then known to have unbounded
   negativity (see \cite[Sect.~2]{Duke}). Indeed,
   taking for $n\in\N$
   the graph $\Gamma_n$ of the Frobenius morphism
   obtained by taking $p^n$-th powers, we have
   $\Gamma_n^2=p^n(2-2g)\to-\infty$
   (see \cite[Ex.~V.1.10]{Hartshorne}).
   By
   Theorem~\ref{denominators-to-BNC}, $X$ must have
   unbounded Zariski
   denominators.
   In the particular case at hand,
   these are in fact quickly detected:
   Denote by $F_2$ a fiber of the second projection $X\to C$,
   and consider the divisor $D_n=F_2+\Gamma_n$.
   The negative part of its Zariski decomposition has support
   $\Gamma_n$ with coefficient
   $$
      \frac{D_n\cdot\Gamma}{\Gamma_n^2}=\frac{1+\Gamma_n^2}{\Gamma_n^2}
   $$
   Since numerator and denominator are coprime for all $n$, we see
   that the Zariski denominator is
   $-\Gamma_n^2=p^n(2g-2)$ and hence tends to infinity.

\end{example}

   Next, we determine concrete bounds on the Zariski denominators
   for classes of surfaces $X$ for which bounded negativity holds and explicit bounds $b(X)$ are known.

\begin{example}[Surfaces with nef anticanonical bundle]
   Let $X$ be a smooth projective surface with $-K_X$ nef.
   As a consequence of the adjunction formula, we have the negativity bound
   $b(X)=2$. Indeed, for every irreducible curve one has
   $2g(C) - 2 = K_{X}\cdot C + C^{2} \le C^{2}$,
   and hence $C^2\ge -2$.

   So, for every pseudo-effective integral divisor $D$ on $X$,
   the Zariski decomposition of $2^{\rho-1}!\cdot D$
   is integral.
\end{example}

\begin{example}[Surfaces with $d(X)=1$]
   Let $X$ be a smooth projective surface, such that all negative curves on
   $X$ are $(-1)$-curves. Then the Zariski decomposition
   of pseudo-effective integral divisors on $X$ is integral.
   Indeed, note that if $X$ contains only
   $(-1)$-curves, then every intersection matrix $S$ of negative
   curves, which are in the support of the negative part of the Zariski
   decomposition of a divisor $D$, has the form
   $-I_{k} = {\rm diag}({-1, \ldots, -1})$. (This follows from negative definiteness.)
   Equation~\eqnref{eqn:ai} shows then that the coefficient of a component $N_i$
   of the negative part of $D$ is the integer $-D\cdot N_i$.

   So we have $d(X)=1$ for instance
   on del Pezzo surfaces
   (see also \cite[Sect.~3]{BKS}),
   and conjecturally
   (according to a weaker form of the SHGH conjecture,
   cf.~\cite{DeFernex})
   for
   all blow-ups of $\P^2$ in several points in very general position.

\end{example}

\begin{example}[Surfaces with large $d(X)$]\label{example:large-d}
   By contrast, note that on the blow-up of $\P^2$ in three collinear points,
   fractional Zariski decompositions occur (see \cite[Example~2.3.20]{PAG}).
   This is in accordance with the fact that a $(-2)$-curve exists on that surface.
   More generally, let $X$ be the blow-up of $r$ points on a line $L$ in $\P^2$
   and denote
   the strict transform of $L$ by $\tilde L$, the pull-back of a general line by $H$
   and the exceptional divisors by $E_1,\dots,E_r$.
   Note first that the only negative curves on $X$ are $E_1,\dots,E_r$ and $\tilde L$.
   (This can be seen quickly by intersecting with $L$ and using a B\'ezout type argument.)
   So we know that
   $$
      b(X)=r-1
      \,.
   $$
   We claim that we have
   $$
      d(X)=b(X)
   $$
   in the case at hand.
   Consider to this end the Zariski decomposition of the divisor
   $\tilde L + H$. The coefficient of $\tilde L$ in its negative part
   is
   $$
      a=\frac{r-2}{r-1}
      \,.
   $$
   Therefore, $d(X)\ge r-1$. On the other hand, if $D$ is an effective
   divisor, then the intersection matrix of its negative part is either
   $-I_k={\rm diag}({-1, \ldots, -1})$ for $1\le k\le r$, or of the form
   $$
   		\left(\begin{array}{cccc}
   				&&&1\\
   				&-I_k&&\vdots\\
   				&&&1\\
   				1&\cdots&1&1-r\\
   		\end{array}\right)
   $$
   for $0\le k\le r-2$. In the latter case it is easy to see by an inductive argument that
   the determinant has absolute value $r-1-k$. Hence, by Equation~\eqnref{eqn:ai},
   we have $d(X)\le r-1$.
\end{example}

\begin{example}[Surfaces with $d(X)$ bigger than $b(X)$]\label{example:bigger}
   In order to see that the occurring denominators can in fact be larger than
   the least self-intersection dictates, consider the blow-up of $\P^2$ in
   $r = k_1+k_2$ points of which $k_1$  lie exclusively on a line $L_1$ and
   $k_2$ exclusively on a second line $L_2$. Assume further $k_1$ and $k_2$ to
   be coprime and both $\ge 4$.
   A computation shows that
   the divisor $H+ \tilde L_1 + \tilde L_2$ has negative part
   supported on $\tilde L_1$ and $\tilde L_2$ with coefficients
   $$
      a_i=\frac{k_1k_2-k_1-k_2-k_i}{k_1k_1-k_1-k_2}
      \,,
   $$
   respectively. In these expressions numerators and denominators are coprime, hence
   the Zariski denominator is $k_1k_2-k_1-k_2$. So we found that
   $$
      d(X)\ge k_1k_2-k_1-k_2
      \,,
   $$
   On the other hand, it is not hard to see that
   $$
      b(X)=\max(k_1-1,k_2-1)
      \,.
   $$
   So we have constructed a series of examples where
   $d(X)$ grows at least quadratically in $r$, whereas $b(X)$
   grows linearly in $r$.
\end{example}

   Note that in Example~\ref{example:bigger} we are still far from the
   theoretical upper bound
   $d(X) \le b(X)^{\rho(X) -1}$
   given by
   Theorem~\ref{BNC-to-denominators}, because $\rho(X)$ also
   grows linearly in $r$ in these cases.
   It would therefore be very interesting to know the answer to
   the following

\begin{question}
   Is there a sequence of surfaces where  $d(X)$ is not bounded by a polynomial in $b(X)$?
\end{question}



\footnotesize
   \bigskip
   Thomas Bauer,
   Fachbereich Mathematik und Informatik,
   Philipps-Universit\"at Marburg,
   Hans-Meerwein-Stra\ss e,
   D-35032 Marburg, Germany.

   \nopagebreak
   \textit{E-mail address:} \texttt{tbauer@mathematik.uni-marburg.de}

   \medskip
   Piotr Pokora,
   Instytut Matematyki,
   Pedagogical University of Cracow,
   Podchor\c a\.zych 2,
   PL-30-084 Krak\'ow, Poland.

   Current Address:
   Fachbereich Mathematik und Informatik,
   Philipps-Universit\"at Marburg,
   Hans-Meerwein-Stra\ss e,
   D-35032 Marburg, Germany.

   \nopagebreak
   \textit{E-mail address:} \texttt{piotrpkr@gmail.com, piotrpokora@daad-alumni.de }

   \medskip
   David Schmitz,
   Fachbereich Mathematik und Informatik,
   Philipps-Universit\"at Marburg,
   Hans-Meerwein-Stra\ss e,
   D-35032 Marburg, Germany.

   \nopagebreak
   \textit{E-mail address:} \texttt{schmitzd@mathematik.uni-marburg.de}



\begin{thebibliography}{99}\footnotesize\compact

\bibitem{Duke}
   Bauer, Th., Harbourne, B., Knutsen, A. L., K\"uronya, A. M\"uller-Stach, S., Roulleau, X., Szemberg, T.:
   Negative curves on algebraic surfaces.
   Duke Math. J. \textbf{162}, 1877--1894 (2013)

\bibitem{BKS}
   Bauer, Th., K\"uronya, A., Szemberg, T.:
   Zariski chambers, volumes, and stable base loci.
   J. Reine Angew. Math. \textbf{576}, 209--233 (2004)

\bibitem{DeFernex}
   De Fernex, T.:
   Negative curves on very general blow-ups of $\P^2$.
   Projective Varieties with Unexpected Properties, a Volume in Memory of Giuseppe Veronese, pp. 199--207, de Gruyter, Berlin, 2005

\bibitem{Fuj79}
   Fujita, T.:
   On Zariski problem.
   Proc. Japan Acad. 55, Ser. A, 106-110 (1979)

\bibitem{Harbourne1}
   Harbourne, B.:
   Global aspects of the geometry of surfaces.
   Annales Universitatis Paedagogicae Cracoviensis Studia Mathematica \textbf{vol. IX}, 5--41 (2010)

\bibitem{Hartshorne}
   Hartshorne, R.:
   Algebraic Geometry.
   Graduate texts in mathematics (52), New York, Springer-Verlag, 1977.

\bibitem{PAG}
   Lazarsfeld, R.:
   \textit{Positivity in Algebraic Geometry I \& II.}
   Ergebnisse der Mathematik und ihrer Grenzgebiete, Vols. 48 \& 49, Springer Verlag, Berlin, 2004.

\bibitem{Zar62}
   Zariski, O.:
   The theorem of Riemann-Roch for high multiples of an effective divisor on an algebraic surface.
   Ann. Math. \textbf{76}, 560-615 (1962)

\end{thebibliography}
\end{document}